\documentclass{amsart}
\usepackage{amsfonts,amssymb,latexsym,amsthm,amsmath}
\usepackage{euscript}
\newtheorem{theorem}{Theorem}

\newtheorem{lemma}[theorem]{Lemma}

\newtheorem{corollary}[theorem]{Corollary}

\newcommand{\ilimit}{\mbox{$\,\displaystyle{\lim_{\longleftarrow}}\,$}}
\usepackage{tikz}

\numberwithin{equation}{section}
\usetikzlibrary{matrix}
\usepackage[arrow,matrix,curve]{xy}\SilentMatrices
\def\xyma{\xymatrix@M.7em}

\begin{document}

\title{Dimension Quotients, Fox Subgroups and Limits of Functors}
\author{Roman Mikhailov and Inder Bir S. Passi}
\maketitle
\begin{abstract}This paper presents a description of the fourth dimension quotient, using the theory of limits of functors from the category of free presentations of a given group to the category of abelian groups. A functorial description of a quotient of the third Fox subgroup is given and, as a consequence, an identification (not involving an isolator) of the third Fox subgroup  is obtained. It is shown that the limit over the category of free representations of the third Fox quotient represents the composite of two derived quadratic functors.
\end{abstract}
\section{Introduction}
Given a group $G$, let $\mathbb Z[G]$ be its integral group ring and $\mathfrak g$ the augmentation ideal of $\mathbb Z[G]$. The dimension quotients of $G$ are defined to be its subquotients $D_n(G)/\gamma_n(G),\ n\geq 1$, where $D_n(G):=G\cap (1+\mathfrak g^n)$ is the $n$th dimension subgroup of $G$ and $\gamma_n(G) $ is the $n$th term in the lower central series $\{\gamma_i(G)\}_{i\geq 1}$ of $G$. The  evaluation  of dimension quotients is a challenging problem in the theory of group rings, and has been a subject of investigation since 1935 (\cite{Cohn:1952}, \cite{Gupta:1987},  \cite{Losey:1960}, \cite{Magnus:1935}, \cite{Magnus:1937}, \cite{MP:2009},  \cite{Passi:1968}, \cite{Passi:1979}). While these quotients are trivial for free groups (\cite{Magnus:1937}, \cite{Witt:1937}),  for $n=1,\,2,\,3$ in case of all groups, and for odd prime-power groups \cite{Passi:1968}, it was first shown in 1972 by E. Rips \cite{Rips:1972} that $D_4(G)/\gamma_4(G)$ is, in general, non-trivial. Subsequently, the structure of these fourth  dimension quotients has been described by K. I. Tahara (\cite{Tahara:1977a}, \cite{Tahara:1977b})  and Narain Gupta \cite{Gupta:1987}. Instances of the non-triviality of dimension quotients in all dimensions $n\geq 4$ are now known (\cite{Gupta:1987}; \cite{MP:2009}, p.\,111); however, their precise structure still remains an open problem.
\par\vspace{.25cm}
Another challenging problem concerning normal subgroups determined by two-sided ideals in group rings is the so-called Fox subgroup problem (\cite{Fox}, page 557; \cite{Birman:1974}, Problem 13; \cite{Gupta:1987}). It asks for the identification of the normal subgroup $F(n,\,R):=F\cap (1+\mathfrak r\mathfrak f^n)$ for a  free group $F$ and its normal subgroup $R$. A solution to this problem has been given by I. A. Yunus \cite{Yunus:1984} and Narain Gupta (\cite{Gupta:1987}, Chapter III). It turns out that while $F(1,\,R)=\gamma_2(R)$, $F(2,\,R)=[R\cap \gamma_2(F),\,R\cap \gamma_2(F)]\gamma_3(R)$, the identification of $F(n,\,R), n\geq 3$, is given as an isolator of a subgroup. For instance, \linebreak $F(3,\,R)=\sqrt{G(3,\,R)}$, where $$G(3,\,R):=\gamma_2(R\cap \gamma_3(F))[[R\cap \gamma_2(F),\,R],\, R\cap \gamma_2(F)]\gamma_4(R).$$ This identification essentially amounts to the one when the coefficients of the group ring are  in the field of rational numbers, rather than in the ring of integers,  and thus raises the question of the precise determination of the involved torsion.
\par\vspace{.25cm}
Our aim in this paper is to present an entirely different approach to the above problems via {derived functors and} limits of functors over the category of free presentations. Our approach is motivated by the connections between the theory of limits of functors with  homology of groups, derived functors in the sense of Dold-Puppe \cite{DP:1961}, cyclic homology  and group rings (\cite{EM:2008}, \cite{MP:2016}, \cite{MP:2017}, \cite{Quillen:1989}). For instance, the even dimensional integral homology groups turn up  as limits \cite{EM:2008}:
$$
H_{2n}(G)=\ilimit (R/\gamma_2(R))_F^{\otimes n},
$$
where $(R/\gamma_2(R))^{\otimes n}$ is the $n$th tensor power of the relation module $R/\gamma_2(R))$, and $(R/\gamma_2(R))_F^{\otimes n}$ is the group of $F$-coinvariants,  the action of $F$ on $(R/\gamma_2(R))^{\otimes n}$, via conjugation in $R$,  being  diagonal. Certain derived functors in the sense of Dold-Puppe \cite{DP:1961} turn out  as limits \cite{MP:2016}:
$$
L_1{\sf SP}^2(G_{ab})=\ilimit \frac{\gamma_2(F)}{\gamma_2(R)\gamma_3(F)},\quad L_1{\sf SP}^3(G_{ab})=\ilimit \frac{\gamma_3(F)}{[\gamma_2(R),\,F]\gamma_4(F)},
$$
where $L_1{\sf SP}^2$ and $L_1{\sf SP}^3$ are the first derived functors of the symmetric square and cube functor respectively  and $G_{ab}:=G/\gamma_2(G)$. The description of derived functors $L_1{\sf SP}^n(G_{ab})$ as limits for $n\geq 4$ is given in \cite{MP:2017}.
 An application of the theory of limits to cyclic homology is given in \cite{Quillen:1989} where it is shown that the cyclic homology of algebras can be defined as limits over the category of free presentations of  certain simply defined functors. We work in the same direction, but consider the category of groups. Our approach brings out a fresh context for the study of  dimension subgroups and Fox subgroups.
\par\vspace{.25cm}To describe the main results of the present  work,
let $F$ be a free group, $R$ a  normal subgroup of $F$, and $G=F/R$. Then there is a natural  short exact sequence
$$
\frac{R\cap \gamma_2(F)}{\gamma_2(R)(R\cap \gamma_4(F))}\hookrightarrow \frac{\gamma_2(F)}{\gamma_2(R)\gamma_4(F)}\twoheadrightarrow \frac{\gamma_2(G)}{\gamma_4(G)}.
$$Observe that the first two terms can be viewed as functors from the category of free presentations $$R\hookrightarrow F\twoheadrightarrow G$$ of $G$ to the category of abelian groups.
The limit functor $\ilimit$ is known to be left \linebreak exact, i.e., it sends monomorphisms to monomorphisms, however, it is not right \linebreak exact, and therefore short exact  sequences of presentations induce long exact\linebreak  sequences involving higher  derived $\ilimit^i$-terms. For instance, the above short \linebreak sequence
induces the following long exact sequence
$$
\ilimit \frac{R\cap \gamma_2(F)}{\gamma_2(R)(R\cap \gamma_4(F))}\hookrightarrow \ilimit \frac{\gamma_2(F)}{\gamma_2(R)\gamma_4(F)}\to \frac{\gamma_2(G)}{\gamma_4(G)}\to \ilimit^1\frac{R\cap \gamma_2(F)}{\gamma_2(R)(R\cap \gamma_4(F))}.
$$
Our main result on dimension quotients describes the cokernel of the left monomorphism in the above exact  sequence.  To be precise, we have
\par\vspace{.25cm}
\begin{theorem}\label{thdim}
There is a natural short exact sequence
$$
\ilimit \frac{R\cap \gamma_2(F)}{\gamma_2(R)(R\cap \gamma_4(F))}\hookrightarrow \ilimit \frac{\gamma_2(F)}{\gamma_2(R)\gamma_4(F)}\twoheadrightarrow \frac{D_4(G)}{\gamma_4(G)}.
$$
\end{theorem}\par\vspace{.25cm}\noindent
Thus we present a description of the  fourth dimension quotient purely in  functorial terms (not involving the group ring), which to us, is a very surprising result.   \par\vspace{.5cm}
We next  give a  functorial description of the quotient $F(3,\,R)/G(3,\,R)$,  together with a complete identification (not involving an isolator),  of the third Fox subgroup $F(3,\,R)$.
\par\vspace{.25cm}
\begin{theorem}\label{fox} Let $F$ be a free group and $R$ a normal subgroup of $F$. \par\vspace{.25cm}\noindent
(a) There is a natural isomorphism
$$
\frac{F(3,\,R)}{G(3,\,R)}\simeq L_1{\sf SP}^2\left(\frac{R\cap \gamma_2(F)}{\gamma_2(R)(R\cap \gamma_3(F))}\right).
$$\\
(b) \quad\quad\quad\quad\quad\quad\quad$
F(3,\,R)=G(3,\,R)W$,\\
where $W$ is a subgroup of $F$, generated by elements
$$
[x,\,y]^m[x,\,s_y]^{-1}[y,\,s_x]
$$
with
\begin{align*}
& x^m=r_xs_x,\ r_x\in R\cap \gamma_3(F),\ s_x\in \gamma_2(R)\\
& y^m=r_ys_y,\ r_y\in R\cap \gamma_3(F),\ s_y\in \gamma_2(R).
\end{align*}
\end{theorem}
\par\vspace{.25cm}
Finally, we give a description of $\ilimit \frac{F(3,\,R)}{G(3,\,R)}$, where it may be noted that  composition of derived functors appears.\par\vspace{.25cm}
\begin{theorem}\label{foxlimit}
There is a natural isomorphism
$$
\ilimit \frac{F(3,R)}{G(3,R)}\simeq L_1{\sf SP}^2\left(L_1{\sf SP}^2(G_{ab})\right).
$$
In particular, there a monomorphism
$$
L_1{\sf SP}^2\left(L_1{\sf SP}^2(G_{ab})\right)\hookrightarrow \frac{F(3,R)}{G(3,R)}.
$$
\end{theorem}

\par\vspace{.25cm}
The paper is organized as follows. In Section 2 - {\it Preliminaries}- we recall the basic properties and results concerning limits of functors and derived functors needed in the sequel. Section 3 - {\it Generalized Dimension Quotients} - is devoted to proving several Lemmas on subgroups determined by two-sided ideals in free group rings.  Theorems 1, 2 and 3 are  proved in Sections 4, 5 and 6 respectively. We refer the reader to \cite{IM:2014} for more details about limits, and to  (\cite{Gupta:1987}, \cite{MP:2009}) for  the identification of normal subgroups determined by  two-sided idelas in group rings.

\par\vspace{.5cm}
\section{Preliminaries}
\subsection{Elementary properties of limits}

We begin by recalling  the  basic facts about limits of functors. Let $\mathcal C,\, \mathcal D$ be arbitrary categories, and let ${\sf Ab}$ be the category of abelian groups. Recall (\cite{Mac_Lane}, Chapter V) that the limit  $\ilimit\ \mathcal F$ of a functor $\mathcal F:\mathcal C\to \mathcal D$ is an object of $\mathcal D$ together with a universal collection of morphisms $$\{\varphi_c:\ilimit\ \mathcal F \to \mathcal F(c)\}_{c\in \mathcal C}\ \text{such that}\  \mathcal F(f) \varphi_c=\varphi_{c'}\text{ for any morphism}\ f:c\to c'.$$ Universality means that for any object $d\in \mathcal D$ and any collection of morphisms $\{\psi_c:d\to \mathcal F(c)\}_{c\in \mathcal C}$ such that $\mathcal F(f) \psi_c=\psi_{c'}$ for any morphism $f:c\to c'$ there exists a unique morphism $\alpha:d\to \ilimit\ \mathcal F$ such that $\psi_c=\varphi_c\alpha.$ The limit of a functor does  not always exist; however,  if it exists, then it is unique up to unique isomorphism that commutes with morphisms $\varphi_c.$

\par\vspace{.25cm}
A category $\mathcal C$ is said to be {\it strongly connected} if for any objects $c,\,c'\in \mathcal C$ the hom-set ${\sf Hom}_{\mathcal C}(c,\,c')$ is non-empty. A standard argument implies that, for a strongly connected category $\mathcal C$ and a functor $\mathcal F: \mathcal C\to {\sf Ab},$ for any object $c\in \mathcal C$, there is a natural embedding
$$\alpha_c:
\ilimit\ \mathcal F\hookrightarrow \mathcal F(c).
$$
 More precisely, $\ilimit\ \mathcal F$ can be identified as follows:
\begin{multline*}\ilimit\ \mathcal F=\{x\in \mathcal F(c)\mid \text{for all}\ c'\in \mathcal C, \ \text{and morphisms}\ f_1,\,f_2:c\to c',\\ \mathcal F(f_1)(x)=\mathcal F(f_2)(x)\}.\end{multline*}
\par\vspace{.25cm}
A {\it category  with pairwise coproducts} is a category $\mathcal C$ such that
for any objects  $c_1,\,c_2\in \mathcal C$ there exists the coproduct $c_1 \buildrel{i_1}\over{\longrightarrow} c_1\sqcup c_2 \buildrel{i_2}\over{\longleftarrow} c_2$ in $\mathcal C$. For a category with pairwise coproducts $\mathcal C$, objects $c_1,\,c_2\in \mathcal C$, and a functor $\mathcal F: \mathcal C\to {\sf Ab}$, there is a natural map
$$
{\sf T}: \mathcal F(c_1)\oplus \mathcal F(c_2)\to \mathcal F(c_1\sqcup c_2).
$$
A functor $\mathcal F$ is called {\it monoadditive}, if $\sf T$ is injective for any pair of objects $c_1,\,c_2\in \mathcal C$.
\par\vspace{.25cm}
The following lemma is due to S. O. Ivanov. It gives a way to define the limit of a functor as an equalizer. This lemma will not be used in the proofs of main statements of the paper.\par\vspace{.25cm}
\begin{lemma}\label{equalizer} Let $\mathcal C$ be a strongly connected category with pairwise coproducts and $\mathcal F:\mathcal C\to {\sf Ab}$ be a functor. Then there are  exact sequences
$$0 \longrightarrow \ilimit\ \mathcal F \overset{\alpha_c}\longrightarrow \mathcal F(c) \xrightarrow{\ \mathcal F(i_1) - \mathcal F(i_2) \ } \mathcal F(c\sqcup c),$$
$$0 \longrightarrow \ilimit\ \mathcal F \longrightarrow \mathcal F(c)\oplus \mathcal F(c) \overset{\sf T}\longrightarrow \mathcal F(c\sqcup c).$$ The map $ \ilimit\ \mathcal F \to \mathcal F(c)\oplus \mathcal F(c)$ is given by $x\mapsto (\alpha_c(x),-\alpha_c(x)),$ where\linebreak  $\alpha_c:\ilimit\ \mathcal F \to \mathcal F(c)$ is the structure morphism of $\ilimit\ \mathcal F$.
\end{lemma}
\begin{proof} Since $\mathcal C$ is strongly connected, the map
$\alpha_c:\ilimit\ \mathcal F\to \mathcal F(c)$ is a monomorphism and its image is equal to the subgroup of $\mathcal C$-invariant elements:
$$\ilimit\ \mathcal F\cong \{ x\in \mathcal F(c)\mid\, \text{for all}\ c', \,\varphi,\,\psi:c\to c' \ \mathcal F(\varphi)(x)=\mathcal F(\psi)(x) \}.$$
We  identify $\ilimit\ \mathcal F$ with this subgroup.\par\vspace{.25cm}
We claim  that ${\rm Ker}({\sf T})=\{(x,\,-x) \mid x\in \ilimit\ \mathcal F\}.$ Consider two arbitrary arrows $\varphi,\,\psi:c\to c'$ and the commutative diagram:
$$\xymatrix{  \mathcal F(c)\ar[r]^{\mathcal F(i_1)- \mathcal F(i_2)} \ar[d]_{\binom{\ 1}{-1}} & \mathcal F(c\sqcup c) \ar@{=}[d] \\
\mathcal F(c)\oplus \mathcal F(c) \ar[r]^{\sf T}\ar[dr]_{(\mathcal F(\varphi ),\, \mathcal F(\psi))} & \mathcal F(c\sqcup c)\ar[d]^{\mathcal F((\varphi,\psi))} \\
 & \mathcal F(c').
}$$
Assume that $(x,\,y)\in {\rm Ker}({\sf T}).$ Then $(x,\,y)\in {\rm Ker}((\mathcal F(\varphi), \mathcal F(\psi))), $ and hence\linebreak  $\mathcal F(\varphi)(x)= -\mathcal F(\psi)(y).$ If we take $\varphi=\psi={\sf id}_c,$ we get $x=-y.$  It follows that for any $c'\in \mathcal C$ and $\varphi,\psi:c\to c'$ we have $\mathcal F(\varphi)(x)=\mathcal F(\psi)(x).$ Thus   $(x,\,y)=(x,\,-x),$ where $x\in \ilimit\ \mathcal F.$
Let  $x\in \ilimit\ \mathcal F.$ If we take $\varphi=i_1:c\to c\sqcup c$ and $\psi=i_2:c\to c\sqcup c,$ we obtain $(x,\,-x)\in {\rm Ker}({\sf T}).$
\par\vspace{.25cm}
Since $\binom{ \ 1}{-1}:\mathcal F(c)\to \mathcal F(c)\oplus \mathcal F(c)$ is a monomorphism and its image contains the kernel of ${\sf T},$ the kernel of $\mathcal F(i_1)-\mathcal F(i_2)$ is equal to $\ilimit \ \mathcal F$ as well.
\end{proof}
\par\vspace{.12cm}
\begin{corollary}
Let $\mathcal C$ be a strongly connected category with pairwise coproducts and let $\mathcal F:\mathcal C\to {\sf Ab}$ be a functor. Then the functor  $\mathcal F$ is monoadditive if and only if $\ilimit\ \mathcal F=0$.
\end{corollary}

Our main concern is the category $\mathcal E$ of free presentations of a given  group $G$. The objects of $\mathcal E$ are surjective homomorphisms $\pi:F\twoheadrightarrow G$ where $F$ is
a free group and  morphisms $f:(\pi_1:F_1\twoheadrightarrow G)\to (\pi_2:F_2\twoheadrightarrow
G)$ are homomorphisms $f:F_1\to F_2$ such that $\pi_1=\pi_2f$. The
category $\mathcal E$ has coproducts given by
$$(\pi_1:F_1\twoheadrightarrow G)\sqcup(\pi_2:F_2\twoheadrightarrow G)=((\pi_1,\pi_2):F_1*F_2\twoheadrightarrow
G).$$
The category $\mathcal E$ is a strongly connected category with pairwise coproducts and therefore, Lemma \ref{equalizer} applies to  functors $\mathcal F: \mathcal E\to {\sf Ab}$. Thus, in particular, the limits $\ilimit\ \mathcal F$ can be identified with corresponding equalizers.
\par\vspace{.25cm}
A representation (i.e., a functor) $\mathcal F(F\twoheadrightarrow G): \mathcal E\to {\sf Ab}$ will be  called a $G$-{\it representation} if, for any $(F\twoheadrightarrow G)\in \mathcal E$, the natural inclusion
$$
\ilimit \mathcal F\hookrightarrow \mathcal F(F\twoheadrightarrow G)
$$
is an isomorphism. That is, a $G$-representation $\mathcal F(F\twoheadrightarrow G)$ depends only on $G$.\par\vspace{.25cm}

The limit $\ilimit$ can be defined as the right adjoint to the diagonal functor. It is left exact, but not right exact. A short exact sequence of representations
$$
\mathcal F_1\hookrightarrow \mathcal F_2\twoheadrightarrow \mathcal F_3
$$
induces a long exact sequence of derived functors $\ilimit^i:$
$$
\ilimit \mathcal F_1\hookrightarrow \ilimit \mathcal F_2\to \ilimit \mathcal F_3\to \ilimit^1\mathcal F_1\to \ilimit^1\mathcal F_2\to \ilimit^1 \mathcal F_3\to \dots
$$
See (Section 2, \cite{IM:2014}) for details. In this paper, we will use only one property of higher limits, namely, their triviality for $G$-representations. The higher limits $\ilimit^i\mathcal F$ can be defined as $i$th cohomology of the category $\mathcal E$ with coefficients in a representation $\mathcal F$. Since the category $\mathcal E$ has pairwise products, it is contractible (see, for example, Lemma 3.5 \cite{IM:2014}), hence for any $G$-representation $\mathcal F$, $\ilimit^i\mathcal F=0,\ i\geq 1$.
\begin{lemma}\label{limlemma}
Let $\mathcal F(F\twoheadrightarrow G)$ be a representation, which lives in an exact sequence
$$
\mathcal F_1(G)\to \mathcal F(F\twoheadrightarrow G)\to \mathcal F_2(G),
$$
where $\mathcal F_i(F\twoheadrightarrow G)=\mathcal F(G),\ i=1,\,2$ are $G$-representations. Then $\mathcal F(F\twoheadrightarrow G)$  is also a $G$-representation.
\end{lemma}
\begin{proof}
Let us first consider the case $\mathcal F_2=0$, i.e,  the case of a natural surjection
$$
\mathcal F_1(G)\twoheadrightarrow \mathcal F(F\twoheadrightarrow G).
$$
The assertion in this case  follows from the following diagram:
$$
\xyma{\ilimit \mathcal F_1\ar@{>->}[r] \ar@{->}[d] & \mathcal F_1(G)\ar@{->>}[r]\ar@{->>}[d] & \frac{\mathcal F_1(G)}{\ilimit \mathcal F_1}\ar@{->>}[d]\\
\ilimit \mathcal F\ar@{>->}[r] & \mathcal F(F\twoheadrightarrow G)\ar@{->>}[r] & \frac{\mathcal F(F\twoheadrightarrow G)}{\ilimit \mathcal F}},
$$
since the quotient $\frac{\mathcal F_1(G)}{\ilimit \mathcal F_1}$ is zero. Dually, if $\mathcal F_1=0$, the quotient
$$
\mathcal F_2(G)/\mathcal F(F\twoheadrightarrow G)
$$
is a $G$-representation, since it is an epimorphic image of a $G$-representation. The assertion in this case  follows from the following commutative diagram
$$
\xyma{\ilimit \mathcal F\ar@{>->}[r]\ar@{>->}[d] & \ilimit \mathcal F_2\ar@{=}[d]\ar@{->}[r] & \ilimit \frac{\mathcal F_2}{\mathcal F(F\twoheadrightarrow G)}\ar@{=}[d]\\
\mathcal F(F\twoheadrightarrow G)\ar@{>->}[r] & \mathcal F_2(G)\ar@{->>}[r] & \frac{\mathcal F_2(G)}{\mathcal F(F\twoheadrightarrow G)}.}
$$
Let us now assume that the map $\mathcal F_1(G)\to \mathcal F(F\twoheadrightarrow G)$ is a monomorphism and $\mathcal F(F\twoheadrightarrow G)\to \mathcal F_2(G)$ is an epimorphism. Since $\ilimit^1 \mathcal F_2=0$, we obtain the \linebreak following diagram
$$
\xyma{\ilimit \mathcal F_1\ar@{>->}[r] \ar@{=}[d] & \ilimit \mathcal F\ar@{->>}[r] \ar@{>->}[d] & \ilimit \mathcal F_2\ar@{=}[d]\\ \mathcal F_1(G)\ar@{>->}[r] & \mathcal F(F\twoheadrightarrow G)\ar@{->>}[r] & \mathcal F_2(G)}
$$
and the result follows.
\end{proof}\par\vspace{.25cm}
\subsection{Quadratic functors}
We will use the following basic quadratic functors: \linebreak tensor square $\otimes^2,$ symmetric square ${\sf SP}^2$, exterior square $\Lambda^2$ and antisymmetric square $\tilde \otimes^2$. Recall that, for an abelian group $A$,
\begin{align*}
& {\sf SP}^2(A)=\otimes^2(A)/\langle a\otimes b-b\otimes a,\ a,\,b\in A\rangle\\
& \Lambda^2(A)=\otimes^2(A)/\langle a\otimes a,\ a\in A\rangle\\
& \tilde \otimes^2(A)=\otimes^2(A)/\langle a\otimes b+b\otimes a,\ a,\,b\in A\rangle.
\end{align*}
For any abelian group $A$, there is a natural exact sequence
$$
0\to A\otimes \mathbb Z/2\to \tilde\otimes^2(A)\to \Lambda^2(A)\to 0,
$$
where the left hand map is given by $\bar a\mapsto a\otimes a,\ a\in A,\ \bar a=a+2A\in A\otimes \mathbb Z/2$.

The derived
functors in the sense of Dold-Puppe \cite{DP:1961} are defined as follows. For an
abelian group $A$ and an endofunctor $F$ on the category of
abelian groups, the derived functor of $F$ is given as
$$
L_iF(A)=\pi_i(FKP_\ast),\ i\geq 0,
$$
where $P_\ast \to A$ is a projective resolution of $A$, and
 $K$ is  the Dold-Kan transform,  inverse to the Moore normalization
 functor from simplicial abelian groups to chain complexes.

The first derived functor $L_1{\sf SP}^2(A)$ which will be used in many places of this paper is the natural quotient
of ${\sf Tor}(A,\,A)$ by diagonal elements $(a,\,a),\ ma=0,\ a\in A$. The first derived functor $L_1\Lambda^2(A)$ is the subfunctor of ${\sf Tor}(A,\,A)$ generated by these diagonal elements, i.e., there is a natural short exact sequence
$$
0\to L_1\Lambda^2(A)\to {\sf Tor}(A,A)\to L_1{\sf SP}^2(A)\to 0.
$$
An abstract value of the abelian group $L_1{\sf SP}^2$ can be easily computed from the following data:
$$
L_1{\sf SP}^2(\mathbb Z/m)=L_1{\sf SP}^2(\mathbb Z)=0,\ m\geq 2;
$$
and
$$
L_1{\sf SP}^2(A\oplus B)=L_1{\sf SP}^2(A)\oplus L_1{\sf SP}^2(B)\oplus {\sf Tor}(A,B),
$$
for all abelian groups $A,B$.
\par\vspace{.25cm}
The derived functors $L_1\Lambda^2$ and $L_1{\sf SP}^2$ naturally appear in the homology of Eilenberg-MacLane spaces. For example, for an abelian $A$, there are natural short exact sequences, which do not split (see \cite{Breen})
\begin{align*}
& 0\to \Lambda^3(A)\to H_3(A)\to L_1\Lambda^2(A)\to 0\\
& 0\to L_1{\sf SP}^2(A)\to H_5K(A,2)\to {\sf Tor}(A,\mathbb Z/2)\to 0,
\end{align*}
where $\Lambda^3$ is the exterior cube. We refer to the thesis of F. Jean \cite{Jean:2002} for the structure of derived functors of higher symmetric powers.
\par\vspace{.25cm}
We also need some  other natural exact sequences, like the following
\begin{equation}\label{l1to}
0\to {\sf Tor}(A,\,\mathbb Z/2)\to L_1\tilde\otimes^2(A)\to L_1\Lambda^2(A)\to 0,
\end{equation}
and
\begin{equation}\label{les}
0\to L_1{\sf SP}^2(A)\to {\sf Tor}(A,\,A)\to L_1\tilde \otimes^2(A)\to \\
{\sf SP}^2(A)\to \otimes^2(A)\to \tilde\otimes^2(A)\to 0.
\end{equation}
Note that the map $L_1\tilde \otimes^2(A)\to {\sf SP}^2(A)$ is, in general, non-zero. There sequences are obtained by deriving the sequences
$-\otimes \mathbb Z/2\to \tilde\otimes^2\twoheadrightarrow\Lambda^2$ and ${\sf SP}^2\to \otimes^2\twoheadrightarrow \tilde\otimes^2$ respectively.
\par\vspace{.25cm}
The following sequence will be used the proofs of our main results several times. For a free abelian group $I$ and its subgroup $J$, there is a natural exact sequence
\begin{equation}\label{l1sp}
0\to L_1{\sf SP}^2(I/J)\to \frac{\Lambda^2(I)}{\Lambda^2(J)}\to I\otimes I/J\to {\sf SP}^2(I/J)\to 0,
\end{equation}
where the image of an element $(x,\,y),\ mx=my\in J,\  x,\,y\in I$, is $mx\wedge y$. For the proof see (Theorem 12, \cite{MP:2016}; Section 3, \cite{MP:2017}). The proof directly follows from the result of K\"ock \cite{Kock:2001} saying that the Koszul-type complex $\Lambda^2(J)\to I\otimes J\to {\sf SP}^2(I)$ represents the element $L{\sf SP}^2(I/J)$ of the derived category of abelian groups.

\par\vspace{.5cm}
\section{Generalized dimension quotients} In analogy with the dimension subgroups $$D_n(G)=F\cap (1+\mathfrak r\mathbb Z[F]+\mathfrak f^n)/\gamma_n(R), \ n\geq 1,$$ when $R$ is a normal subgroup of a free group $F$ with $G=F/R$, the normal subgroups $F\cap(1+\mathfrak a+\mathfrak f^n)$, $n\geq 1$, where $\mathfrak a$ is a two-sided ideal of $\mathbb Z[F]$ are called {\it generalized dimension subgroups.} We set $$D(n,\, \mathfrak a):=F\cap (1+\mathfrak a+\mathfrak f^n).$$
An example of description of a generalized dimension subgroup and its connection to a derived functor, which we will use later, is the following. It is shown in \cite{HMP:2009} that, there is a natural isomorphism
\begin{equation}\label{3fr}
\frac{D(3,\,{\mathfrak r\mathfrak f})}{\gamma_2(R)\gamma_3(F)}=\frac{D(3,\,{\mathfrak f\mathfrak r})}{\gamma_2(R)\gamma_3(F)}=L_1{\sf SP}^2(G_{ab}).
\end{equation}
See \cite{MP:2016} and \cite{MP:2017} for more examples of such type.
\par\vspace{.25cm}
We need the identification of certain generalized dimension subgroups.  Recall that \begin{quote}{\it if $R$ is a normal subgroup of a free group $F$ and $Y$ is a basis of $R$, then the two-sided ideal $\mathfrak r\mathbb Z[F]$, viewed as a left (resp. right) $\mathbb Z[F]$-module, is free with basis $\{y-1\,|\,y\in Y\}$} (\cite{Gruenberg:1970}, Theorem 1, p.\,32). \end{quote}

\par\vspace{.25cm}
\begin{lemma}\label{fg1}
 If  $F$ is a free group of finite rank $m\geq 1$, $X=\{x_1,\,x_2,\ldots,\,x_m\}$ an ordered basis of $F$, and $R$  a normal subgroup of $F$ generated by $$\{x_1^{e_1}\xi_1,\, \, x_2^{e_2}\xi_2,\,\ldots,\,x_m^{e_m}\xi_m,\, \xi_{m+1},\ldots\}$$ with $\xi_i\in \gamma_2(F)$ for $i=1,\,2,\ldots,\ $ and integers $e_i\geq 0$ satisfying $e_m|e_{m-1}|\ldots|e_2|e_1$, then the generalized  dimension subgroup
$$D(4,\, \mathfrak f\mathfrak r\mathfrak f+\mathfrak f^2\mathfrak r):=F\cap (1+\mathfrak f\mathfrak r\mathfrak f+\mathfrak f^2\mathfrak r+\mathfrak f^4)$$ is generated, modulo $\gamma_4(F)$,  by the commutators
\begin{itemize}
\item$ [[x_j,\,x_i],\,x_i]^{e_i}\ (j>i),$
\item
$ [[x_j,\,x_i],\,x_k]^{\ell_{jk}}\  (j>i<k), \ \text{where}\  \ell_{jk}=\operatorname{lcm}(e_j,\ e_k).$
\end{itemize}
\end{lemma}
\begin{proof}
It is easy to see that the elements $$ [[x_j,\,x_i],\,x_i]^{e_i}\ (j>i),\ \text{and}\  [[x_j,\,x_i],\,x_k]^{\ell_{jk}}\  (j>i<k), $$where $  \ell_{jk}=\operatorname{lcm}(e_j,\ e_k)$ all lie in $D(4,\, \mathfrak f\mathfrak r\mathfrak f+\mathfrak f^2\mathfrak r)$.\par\vspace{.5cm}\noindent
Let $S=\langle x^{e_i} \ (1\le i\leq m), \gamma_2(F)\rangle$, and $w\in D(4,\,\mathfrak f\mathfrak r\mathfrak f+\mathfrak f^2\mathfrak r)$. Observe that
$$\mathfrak f\mathfrak r\mathfrak f+\mathfrak f^2\mathfrak r+\mathfrak f^4=\mathfrak f\mathfrak s\mathfrak f+\mathfrak f^2\mathfrak s+\mathfrak f^4.$$
Since $D_3(F)=\gamma_3(F)$ and $\{[[x_j,\,x_i],\,x_k]\,|\, m\geq j>i\leq k\leq m\}$ is a basis of $\gamma_3(F)/\gamma_4(F)$, we have, modulo $\gamma_4(F)$,  $$w=\prod _{i=1}^mw_i,\ w_i=\prod _{j>i\leq k}[[x_j,\,x_i],\,x_k]^{a_{ijk}}\in D(4,\,\mathfrak f\mathfrak s\mathfrak f+\mathfrak f^2\mathfrak s),\ a_{ijk}\in \mathbb Z.$$
Modulo $\mathfrak f^4$, we have \begin{multline}\label{1}W_i:=w_i-1=\sum_{j> i\leq k}\{a_{ijk}[(x_j-1)(x_i-1)-(x_i-1)(x_j-1)](x_k-1)\\-(x_k-1)[(x_j-1)(x_i-1)-(x_i-1)(x_j-1)]\}.\end{multline}
Left differentiating  (in the sense of free differential calculus  {[see \cite{Gupta:1987}, p.7]}) with respect to $x_i$, we have\begin{multline}\label{2}_{x_i}W_i=\sum_{j>i}a_{iji}[(x_j-1)(x_i-1)-2(x_i-1)(x_j-1)]\\-\sum_{j>i<k}a_{ijk}(x_k-1)(x_j-1)\in \mathfrak f\mathfrak s+e_i\mathfrak f^2+\mathfrak f^3.\end{multline}
Right differentiating with respect to $x_i$, gives
$$2a_{iji}(x_j-1)\in \mathfrak s+e_i\mathfrak f+\mathfrak f^2\ (j>i).$$Hence $e_j|2a_{iji},\ j>i,$
and Eq(\ref{2}) implies \begin{multline}\label{3}
\sum_{j>i}a_{iji}(x_j-1)(x_i-1)-\sum_{j>i<k}a_{ijk}(x_k-1)(x_j-1)\in \mathfrak f\mathfrak s+e_i\mathfrak f^2+\mathfrak f^3.\end{multline}
Since the second sum does not involve $x_i$, we conclude that \begin{equation}\label{4}\sum_{j>i}a_{iji}(x_j-1)(x_i-1)\in \mathfrak f\mathfrak s+e_i\mathfrak f^2+\mathfrak f^3\end{equation}and \begin{equation}\label{5}\sum_{j>i<k}a_{ijk}(x_k-1)(x_j-1)\in \mathfrak f\mathfrak s+e_i\mathfrak f^2+\mathfrak f^3.\end{equation}
Eq(\ref{4}) implies that $a_{iji}(x_i-1)\in \mathfrak s+e_i\mathfrak f+\mathfrak f^2$, and therefore we have \begin{equation}\label{6}e_i|a_{iji},\ j>i.\end{equation}
Similarly Eq(\ref{5}) implies that\begin{equation}\label{7}e_j|a_{ijk},\ j>i<k.\end{equation}
Eq(\ref{1}) thus reduces to \begin{equation}\label{8}\sum_{j>i<k}a_{ijk}(x_j-1)(x_i-1)(x_k-1)\in \mathfrak f\mathfrak s\mathfrak f+\mathfrak f^2\mathfrak s+\mathfrak f^4.\end{equation}Eq(\ref{8}) yields that \begin{equation}\label{9}e_k|a_{ijk},\ j>i<k.\end{equation}Eqs (\ref{7}) and (\ref{9}) imply \begin{equation}\label{10} \ell_{ij}|a_{ijk},\ j>i<k, \ \ell=\operatorname{lcm}(e_j,\,e_k).\end{equation} Hence the element $w$ lies in the subgroup generated by the commutators \begin{itemize}
\item$ [[x_j,\,x_i],\,x_i]^{e_i}\ (j>i),$
\item
$ [[x_j,\,x_i],\,x_k]^{\ell_{jk}}\  (j>i<k), \ \text{where}\  \ell_{jk}=\operatorname{lcm}(e_j,\ e_k),$ as claimed.\end{itemize}\end{proof}

\begin{corollary}\label{corfrf} If $F$ is a free group and $R$ a normal subgroup of $F$ such that  $F/R\gamma_2(F)$ is  torsion-free, then
$$
D(4,\,{\mathfrak{frf}}+{\mathfrak{f}^2\mathfrak r})=[[R,\,F],\,R]\gamma_4(F).
$$
\end{corollary}
\begin{proof} It is easy to see that $[[R,\,F],\,R]\gamma_4(F)\subseteq D(4,\,{\mathfrak{frf}}+{\mathfrak{f}^2\mathfrak r})$. Observe that, for the reverse inclusion, it suffices to consider the case when $F$ is finitely generated, and so Lemma \ref{fg1} applies.
\end{proof}
 \par\vspace{.25cm} With free group $F$ and its subgroup $S$ as in Lemma \ref{fg1}, we have \par\vspace{.25cm}
\begin{lemma}\label{l2}
$$D(4,\,\mathfrak s\mathfrak f\cap \mathfrak f^3)=[\gamma_2(F),\,S]\gamma_4(F).$$
\end{lemma}
\begin{proof}
Let $w\in D(4,\,\mathfrak s\mathfrak f\cap \mathfrak f^3)$. Then, modulo $\gamma_4(F)$,  $$w=\prod _{i=1}^{m-1}w_i,\ w_i=\prod_{m\geq j>i\leq k\leq m}[[x_j,\,x_i],\,x_k]^{a_{ijk}}, \ a_{ijk}\in \mathbb Z,$$ with $w_i\in D(4,\,\mathfrak s\mathfrak f\cap \mathfrak f^3)$. Modulo  $ \mathfrak s\mathfrak f\cap \mathfrak f^3+\mathfrak f^4$, for every $i\in \{1,\,2,\,\ldots\,,\,m-1\}$,
\begin{multline}\label{11}W_i:=w_i-1=\sum_{m\geq j> i\leq k\leq m}\{a_{ijk}[(x_j-1)(x_i-1)-(x_i-1)(x_j-1)](x_k-1)\\-(x_k-1)[(x_j-1)(x_i-1)-(x_i-1)(x_j-1)]\}\end{multline}
Since $(x_j-1)(x_i-1)-(x_i-1)(x_j-1)\in \mathfrak s$, it follows that
\begin{multline}\label{17}W_i':=\sum_{m\geq j>i\leq k\leq m}a_{ijk}(x_k-1)[(x_j-1)(x_i-1)-(x_i-1)(x_j-1)]\in \mathfrak s\mathfrak f\cap \mathfrak f^3+\mathfrak f^4.\end{multline}
Eq(\ref{17})  implies that for every $j\,\in \{i+1,\,\ldots,\,m\}$ and $i\in \{1,\,2,\,\ldots\,m-1\}$,
\begin{equation}\label{19}
_{x_j}W'_i=\sum_{k=i}^ma_{ijk}(x_k-1)(x_i-1)\in  \mathfrak s+\mathfrak f^3.\end{equation}
It follows that
\begin{equation}\label{20}
e_k|a_{ijk}, \ \text{for}\ m\geq j>i\leq k\leq  m.
\end{equation}
Thus we see that $w_i\in D(4,\,\mathfrak s\mathfrak f\cap \mathfrak f^3)$ only if
\begin{equation}
e_k|a_{ijk}\ (m\ge j>i\leq k\leq m),
\end{equation}and so $w_i$ and hence $w\in [\gamma_2(F),\,S]\gamma_4(F)$.
Hence $D(4,\,\mathfrak s\mathfrak f\cap \mathfrak f^3)\subseteq [\gamma_2(F),\,S]\gamma_4(F).$ The reverse inclusion is easily seen to hold.
\end{proof}
\par\vspace{.25cm}

\begin{lemma} If $F$ is a free group and $R$ a normal subgroup of $F$, then, modulo $\gamma_4(F)$, \\\\
$ D(4,\,\mathfrak f\mathfrak r\cap \mathfrak f^3)=
D(4,\,\mathfrak r\mathfrak f\cap \mathfrak f^3)\\\\=\langle [f',\,f]^e\,|\, f'\in \gamma_2(F), \, f\in F,\, f'^e\in (R\cap \gamma_2(F))\gamma_3(F),\ f^e\in R\gamma_2(F)\rangle.$
\end{lemma}
\begin{proof}
The first equality is an immediate consequence of the canonical \linebreak anti-isomorphism of $\mathbb Z[F]$ induced by $w\mapsto w^{-1},\ w\in F$.\par\vspace{.25cm}
It is easy to check that
$$[f',\,f]^e\in D(4,\,\mathfrak r\mathfrak f\cap \mathfrak f^3)$$ if $ f'\in \gamma_2(F), \, f\in F,\, f'^e\in (R\cap \gamma_2(F))\gamma_3(F),\ f^e\in R\gamma_2(F).$
\par\vspace{.25cm} Conversely, let  $w\in D(4,\,\mathfrak r\mathfrak f\cap \mathfrak f^3).$  To analyze $w$, we may clearly assume that $F$ is finitely generated, $R,\ S$ are as in Lemma \ref{fg1}, and so  $w\in D(4, \,\mathfrak s\mathfrak f\cap \mathfrak f^3)$. Therefore, by Lemma \ref{l2}, we have, modulo $\gamma_4(F), $  $$w=\prod_{i=1}^{m-1}w_i$$ with
$$w_i=\prod_{m\ge j<i\leq k\leq m}[[x_j,\,x_i],\,x_k]^{e_kb_{ijk}},\ b_{ijk}\in \mathbb Z.$$ On collecting terms, we have, $$w\equiv \prod_{k=1}^m[f_k,\,x_k]^{e_k},\ f_k\in \gamma_2(F),\ 1\leq k\leq m.$$
Now $w-1\in \mathfrak r\mathfrak f\cap \mathfrak f^3+\mathfrak f^4$. Therefore we have $$\sum_{k=1}^me_k(f_k-1)(x_k-1)\ \in \mathfrak r\mathfrak f\cap \mathfrak f^3+\mathfrak f^4.$$ Differentiating with respect to $x_k$, yields $$f_k^{e_k}-1\in \mathfrak r+\mathfrak f^3.$$ Hence $$f_k^{e_k}\in (R\cap \gamma_2(F))\gamma_3(F),\ 1\leq k\leq m,$$and consequently  $$w\in \langle [f',\,f]^e\,|\, f'\in \gamma_2(F), \, f\in F,\, f'^e\in (R\cap \gamma_2(F))\gamma_3(F),\ f^e\in R\gamma_2(F)\rangle.$$
\end{proof}
\par\vspace{.25cm}
Observe that, for every free group $F$ and its normal subgroup $R$,
\begin{equation}\label{eqr}
\gamma_2(R)\cap\gamma_3(F)=[R\cap \gamma_2(F),\,R].
\end{equation}
To see this equality, consider the map between exterior squares
\begin{equation}\label{mapl2}
\Lambda^2(R/(R\cap F'))\to \Lambda^2(F_{ab}),
\end{equation}
induced by inclusion $R/(R\cap \gamma_2(F))\hookrightarrow F_{ab}$, where $F_{ab}$ denotes the abelianization of $F$.
The map (\ref{mapl2}) is a monomorphism, since it is induced by a monomorphism of free abelian groups. The needed equality (\ref{eqr}) now follows from the following identifications:
$$
\Lambda^2(R/(R\cap \gamma_2(F)))=\frac{\gamma_2(R)}{[R\cap\gamma_2( F),\,R]},\quad \Lambda^2(F_{ab})=\gamma_2(F)/\gamma_3(F).
$$
Thus, in particular,
$$
(\gamma_2(R)\cap \gamma_3(F))\gamma_4(F)=[R\cap \gamma_2(F),\,R]\gamma_4(F).
$$
In view of the known structure (\cite{Mac_Lane}, Chapter V, Section\,5) of
${\sf Tor}(A,\,B)$ for abelian groups $A,\,B$, the  preceding Lemma immediately yields  the following result.\par\vspace{.25cm}
\begin{lemma}\label{2lemma}
There is a natural epimorphism
$$
\phi: {\sf Tor}(\gamma_2(G)/\gamma_3(G), \,G_{ab})\to \frac{D(4,\,{\mathfrak f\mathfrak r}\cap {\mathfrak  f}^3)}{[R\cap \gamma_2(F),\,R]\gamma_4(F)}=\frac{D(4,\,{\mathfrak f\mathfrak r}\cap {\mathfrak  f}^3)}{(\gamma_2(R)\cap \gamma_3(F))\gamma_4(F)}.
$$
\end{lemma}
\vspace{.5cm}
\section{Proof of theorem \ref{thdim}}
First observe that the cokernel of the natural map
$$
\frac{R\cap D(4,{\mathfrak  f\mathfrak r})}{\gamma_2(R)
\gamma_4(F)}\hookrightarrow \frac{D(4,\,{\mathfrak  f\mathfrak r})}{\gamma_2(R)
\gamma_4(F)}
$$
can be naturally identified with the fourth dimension quotient $$\frac{D_4(G)}{\gamma_4(G)}=\frac{D(4,\,\mathfrak  r\mathbb Z[F])}{R\gamma_4(F)}$$(see, for example,  \cite{Gupta:1987}, p.\. 80). There is an obvious short exact sequence
$$
\frac{R\cap D(4,\,{\mathfrak  f\mathfrak r})}{\gamma_2(R)
\gamma_4(F)}\hookrightarrow \frac{D(4,{\mathfrak  f\mathfrak r})}{\gamma_2(R)
\gamma_4(F)}\to \frac{D_4(G)}{\gamma_4(G)}.
$$
The right hand map is, in fact, surjective. For, let $w\in D(4, \,{\mathfrak  r\mathbb Z[F]})$ so that $w-1\in {\mathfrak  r}\mathbb Z[F]+{\mathfrak  f}^4$. The quotient $\frac{\mathfrak  r\mathbb Z[F]}{\mathfrak  f\mathfrak r}$ is the relation module $R/\gamma_2(R)$, and so  there is a natural epimorphism
$$
R/\gamma_2(R)\twoheadrightarrow \frac{{\mathfrak  r}\mathbb Z[F]+{\mathfrak f}^4}{{\mathfrak  f\mathfrak r}+{\mathfrak  f}^4}
$$
and we can find an element $r\in R$, such that $wr-1\in {\mathfrak  f\mathfrak r}+{\mathfrak  f}^4$.
It thus follows  that every element of the dimension quotient $\frac{D_4(G)}{\gamma_4(G)}$ has a preimage in $\frac{D(4,\,\mathfrak f\mathfrak r)}{\gamma_2(R)
\gamma_4(F)}$.
\par\vspace{.25cm}
Consider the natural diagram with exact rows and columns
\begin{equation}\label{1dia}
\xyma{\frac{R\cap D(4,\,\mathfrak f\mathfrak r)}{\gamma_2(R)
\gamma_4(F)}\ar@{>->}[r]\ar@{>->}[d] &\frac{R\cap \gamma_2(F)}{\gamma_2(R)(R\cap \gamma_4(F))}\ar@{->}[r]\ar@{>->}[d] & \frac{{\mathfrak f}^2}{\mathfrak f\mathfrak r+{\mathfrak f}^4}\ar@{->>}[r]\ar@{=}[d] & \frac{{\mathfrak f}^2}{(R\cap \gamma_2(F)-1)+\mathfrak f\mathfrak r+{\mathfrak f}^4}\ar@{->>}[d]\\
\frac{D(4,\,\mathfrak f\mathfrak r)}{\gamma_2(R)
\gamma_4(F)}\ar@{>->}[r]\ar@{->>}[d] &\frac{\gamma_2(F)}{\gamma_2(R)\gamma_4(F)}\ar@{->}[r] \ar@{->>}[d] & \frac{{\mathfrak f}^2}{\mathfrak f\mathfrak r+{\mathfrak f}^4}\ar@{->>}[r] & \frac{{\mathfrak f}^2}{(\gamma_2(F)-1)+\mathfrak f\mathfrak r+{\mathfrak f}^4}\\
\frac{D_4(G)}{\gamma_4(G)}\ar@{>->}[r] & \frac{\gamma_2(G)}{\gamma_4(G)}
}
\end{equation}
Now observe that
$$
\ilimit \frac{{\mathfrak f}^2}{\mathfrak f\mathfrak r+{\mathfrak f}^4}=\ilimit \frac{\mathfrak f}{{\mathfrak f}^3}\otimes_{\mathbb Z[G]} \frac{\mathfrak f}{{\mathfrak r}+{\mathfrak f}^3}=0.
$$
This follows from monoadditivity of $\frac{\mathfrak f}{{\mathfrak f}^3}\otimes_{\mathbb Z[G]} \frac{\mathfrak f}{{\mathfrak r}+{\mathfrak f}^3}=\frac{\mathfrak f}{{\mathfrak f}^3}\otimes_{\mathbb Z[G]} \frac{\mathfrak g}{{\mathfrak g}^3}$ which can be easily checked.
The top two horizontal exact sequences in (\ref{1dia}) imply the natural isomorphisms
\begin{align*}
& \ilimit \frac{R\cap D(4,\,\mathfrak f\mathfrak r)}{\gamma_2(R)
\gamma_4(F)}\simeq\ilimit \frac{R\cap \gamma_2(F)}{\gamma_2(R)(R\cap \gamma_4(F))}\\
& \ilimit\frac{D(4,\mathfrak f\mathfrak r)}{\gamma_2(R)
\gamma_4(F)}\simeq \ilimit \frac{\gamma_2(F)}{\gamma_2(R)\gamma_4(F)}.
\end{align*}
Therefore, the left hand vertical exact sequence in diagram (\ref{1dia}) implies that we have the following the long exact sequence
\begin{equation}\label{seq1}
\ilimit \frac{R\cap \gamma_2(F)}{\gamma_2(R)(R\cap \gamma_4(F))}\hookrightarrow \ilimit \frac{\gamma_2(F)}{\gamma_2(R)\gamma_4(F)}\to\frac{D_4(G)}{\gamma_4(G)}\to \ilimit^1 \frac{R\cap D(4,\,\mathfrak f\mathfrak r)}{\gamma_2(R)
\gamma_4(F)}
\end{equation}
Next consider the following diagram with exact rows and columns
\begin{equation}\label{2dia}
\xyma{\frac{D(4,\,\mathfrak f\mathfrak r\cap {\mathfrak f}^3)}{(\gamma_2(R)\cap \gamma_3(F))\gamma_4(F)}\ar@{>->}[r]\ar@{>->}[d] & \frac{D(4,\,\mathfrak f\mathfrak r)}{\gamma_2(R)\gamma_4(F)}\ar@{->}[r]\ar@{>->}[d] & \frac{D(3,\,\mathfrak f\mathfrak r)}{\gamma_2(R)\gamma_3(F)}\ar@{>->}[d]\\
\frac{\gamma_3(F)}{(\gamma_2(R)\cap \gamma_3(F))\gamma_4(F)}\ar@{>->}[r]\ar@{->}[d] & \frac{\gamma_2(F)}{\gamma_2(R)\gamma_4(F)}\ar@{->>}[r] \ar@{->}[d] & \frac{\gamma_2(F)}{\gamma_2(R)\gamma_3(F)}\ar@{->}[d]\\
\frac{{\mathfrak f}^3}{\mathfrak f\mathfrak r\cap {\mathfrak f}^3+{\mathfrak f}^4}\ar@{>->}[r] & \frac{{\mathfrak f}^2}{\mathfrak f\mathfrak r+{\mathfrak f}^4}\ar@{->>}[r] & \frac{{\mathfrak f}^2}{\mathfrak f\mathfrak r+{\mathfrak f}^3}}
\end{equation}
(observe that, the right hand horizontal map $\frac{D(4,\,\mathfrak f\mathfrak r)}{\gamma_2(R)\gamma_4(F)}\to \frac{D(3,\,\mathfrak f\mathfrak r)}{\gamma_2(R)\gamma_3(F)}$ is not, in general, surjective). The generalized dimension quotient $\frac{D(3,\,\mathfrak f\mathfrak r)}{\gamma_2(R)\gamma_3(F)}$ is identified with the first derived functor of the symmetric square:
$$
\frac{D(3,\,\mathfrak f\mathfrak r)}{\gamma_2(R)\gamma_3(F)}=L_1{\sf SP}^2(G_{ab})
$$
(see (\ref{3fr})). By Lemmas \ref{limlemma} and \ref{2lemma}, the dimension quotient
$$
\frac{D(4,\,\mathfrak f\mathfrak r\cap {\mathfrak f}^3)}{(\gamma_2(R)\cap \gamma_3(F))\gamma_4(F)}=\frac{D(4,\,\mathfrak f\mathfrak r\cap {\mathfrak f}^3)}{[R\cap \gamma_2(F),\,R]\gamma_4(F)}
$$
is a $G$-representation. Therefore, by Lemma \ref{limlemma}, using the upper horizontal exact sequence in the diagram (\ref{2dia}), we conclude that $\frac{D(4,\,\mathfrak f\mathfrak r)}{\gamma_2(R)\gamma_4(F)}$ is a $G$-representation, and
$$
\ilimit \frac{D(4,\,\mathfrak f\mathfrak r)}{\gamma_2(R)\gamma_4(F)}=\frac{D(4,\,\mathfrak f\mathfrak r)}{\gamma_2(R)\gamma_4(F)}.
$$
Looking at the left hand vertical epimorphism  in the diagram (\ref{1dia}), we conclude that the natural map
$$
\ilimit \frac{D(4,\,\mathfrak f\mathfrak r)}{\gamma_2(R)\gamma_4(F)}=\ilimit \frac{\gamma_2(F)}{\gamma_2(R)\gamma_4(F)}\to \frac{D_4(G)}{\gamma_4(G)}
$$
is an epimorphism, that is, the map $\frac{D_4(G)}{\gamma_4(G)}\to \ilimit^1 \frac{R\cap D(4,\,\mathfrak f\mathfrak r)}{\gamma_2(R)
\gamma_4(F)}$ in (\ref{seq1}) is zero and the asserted short exact sequence follows, and the proof is complete.
\par\vspace{.5cm}
\section{Proof of theorem \ref{fox}}\par\noindent
(a) Let us set
$G(2,R):=[R\cap \gamma_2(F),\,R\cap \gamma_2(F)]\gamma_3(R).$ Since $F(2,R)=G(2,R)$, $F(3,R)$ is a subgroup of $G(2,R)$.
Observe that, in view of (\ref{eqr}), we have
$$
\gamma_2(R\cap \gamma_2(F))\cap \gamma_3(R)=[\gamma_2(R),\, R\cap \gamma_2(F)],
$$
and
\begin{multline*}
\gamma_2(R\cap \gamma_3(F))\cap \gamma_3(R)=[\gamma_2(R)\cap \gamma_3(F),\,R\cap \gamma_3(F)]=\\ [[R\cap \gamma_2(F),\,R],\,R\cap \gamma_3(F)]\subseteq [[R\cap \gamma_2(F),\,R],\,R\cap \gamma_2(F)].
\end{multline*}
We have the following natural diagram with exact rows and columns\par\vspace{.25cm}
\begin{equation}\label{foxdia}
\xyma{& \frac{F(3,\,R)}{G(3,\,R)}\ar@{>->}[d] \ar@{->}[r] &
\frac{\gamma_2(R\cap \gamma_2(F))\cap (1+\mathfrak r\mathfrak f^3+{\mathfrak  r^3\mathbb Z[F]})}
{\gamma_2(R\cap \gamma_3(F))[\gamma_2(R),\,R\cap \gamma_2(F)]}\ar@{>->}[d]\\
\frac{\gamma_3(R)}{[R\cap \gamma_2(F),\,R,\,R\cap \gamma_2(F)]\gamma_4(R)}\ar@{>->}[r]\ar@{->}[d] & \frac{G(2,\,R)}{G(3,\,R)} \ar@{->>}[r] \ar@{->}[d] & \frac{\gamma_2(R\cap \gamma_2(F))}{\gamma_2(R\cap \gamma_3(F))[\gamma_2(R),\,R\cap \gamma_2(F)]}\ar@{->}[d] \\ \frac{{\mathfrak r}^3\mathbb Z[F]}{{\mathfrak r}^3\mathbb Z[F]\cap {\mathfrak r\mathfrak f}^3} \ar@{>->}[r] & \frac{{\mathfrak  r\mathfrak f}^2}{{\mathfrak  r\mathfrak f}^3}\ar@{->>}[r] & \frac{{\mathfrak  r\mathfrak f}^2}{{\mathfrak  r\mathfrak f}^3+{\mathfrak r}^3\mathbb Z[F]}}
\end{equation}

Let us set  $$I:=\frac{R\cap \gamma_2(F)}{\gamma_2(R)},\quad J:=\frac{R\cap \gamma_3(F)}{\gamma_2(R)\cap \gamma_3(F)}.$$ Clearly  $J\subset I\subset R/\gamma_2(R)$. Observe that
$$
\frac{\Lambda^2(I)}{\Lambda^2(J)}=\frac{\gamma_2(R\cap \gamma_2(F))}{\gamma_2(R\cap \gamma_3(F))[\gamma_2(R),\,R\cap \gamma_2(F)]}.
$$
Invoking  the exact sequence (\ref{l1sp}), we have the following sequence
$$
L_1{\sf SP^2}(I/J)\hookrightarrow \frac{\Lambda^2(I)}{\Lambda^2(J)}\to I\otimes I/J.
$$
There is a natural isomorphism
$$
\frac{{\mathfrak  r\mathfrak f}^2}{{\mathfrak  r\mathfrak f}^3+{\mathfrak  r}^3\mathbb Z[F]}=\frac{\mathfrak  r\mathbb Z[F]}{\mathfrak  r\mathfrak f}\otimes \frac{{\mathfrak f}^2}{{\mathfrak f}^3+{\mathfrak  r}^2}
$$
[Observe that, ${\mathfrak  r\mathfrak f}^3+{\mathfrak  r}^3={\mathfrak  r\mathfrak f}^3+{\mathfrak  r}^3\mathbb Z[F],$ and we can omit the last $\mathbb Z[F]$ term.]
Recall that $D(3,\,\mathfrak r^2)=\gamma_2(R)\gamma_3(F)$ (see \cite{MP:2016}). Hence, there are monomorphisms
$$
I\hookrightarrow R/\gamma_2(R)=\frac{\mathfrak  r\mathbb Z[F]}{\mathfrak  r\mathfrak f},\quad I/J\hookrightarrow \frac{{\mathfrak f}^2}{{\mathfrak f}^3+{\mathfrak  r}^2}.
$$
Since both $R/\gamma_2(R)$ and ${\sf Coker}(I\hookrightarrow R/\gamma_2(R))=\frac{R}{R\cap \gamma_2(F)}$ are free abelian, we see that there is a natural monomorphism
$$
I\otimes I/J\hookrightarrow \frac{{\mathfrak  r\mathfrak f}^2}{{\mathfrak  r\mathfrak f}^3+{\mathfrak r}^3}.
$$
The sequence (\ref{l1sp}) implies the following diagram
$$
\xyma{L_1{\sf SP}^2(I/J)\ar@{>->}[r] & \frac{\Lambda^2(I)}{\Lambda^2(J)}\ar@{->}[r]\ar@{=}[d] & I\otimes I/J\ar@{>->}[d]\\ & \frac{\gamma_2(R\cap \gamma_2(F))}{\gamma_2(R\cap \gamma_3(F))[\gamma_2(R),\,R\cap \gamma_2(F)]} \ar@{->}[r] & \frac{{\mathfrak  r\mathfrak f}^2}{{\mathfrak  r\mathfrak f}^3+{\mathfrak r}^3}}
$$
We thus obtain the following identification:
$$
\frac{\gamma_2(R\cap \gamma_2(F))\cap (1+{\mathfrak  r\mathfrak f}^3+{\mathfrak  r}^3)}{\gamma_2(R\cap \gamma_3(F))[\gamma_2(R),\,R\cap \gamma_2(F)]}=L_1{\sf SP}^2(I/J).
$$
 There is a simple way to pick representatives of $L_1{\sf SP}^2(I/J)$ in $\gamma_2(R\cap \gamma_2(F))$, which also follows from the sequence (\ref{l1sp}). The subgroup $L_1{\sf SP}^2(I/J)$ is generated by elements
$$
[x,\,y]^m,\ x,\,y\in R\cap \gamma_2(F),\ x^m, \,y^m\in (R\cap \gamma_3(F))\gamma_2(R).
$$
One can easily check that, for such a pair $x,\,y,$
$$
[x,\,y]^m-1\in {\mathfrak r\mathfrak f}^3+{\mathfrak  r}^3.
$$
We assert that the horizontal arrow (let us call it $q$) in  the diagram (\ref{foxdia})
$$q: \frac{F(3,\,R)}{G(3,\,R)}\to \frac{\gamma_2(R\cap \gamma_2(F))\cap (1+{\mathfrak  r\mathfrak f}^3+{\mathfrak  r}^3)}{\gamma_2(R\cap \gamma_3(F))[\gamma_2(R),\,R\cap \gamma_2(F)]}$$
is an epimorphism. For, let
\begin{align*}
& x^m=r_xs_x,\ r_x\in R\cap \gamma_3(F),\ s_x\in \gamma_2(R),\\
& y^m=r_ys_y,\ r_y\in R\cap \gamma_3(F),\ s_y\in \gamma_2(R).
\end{align*}
Consider the element
$$
w:=[x,\,y]^m[x,\,s_y]^{-1}[y,\,s_x].
$$
Clearly, $q(w)=[x,\,y]^m$, since $[x,\,s_y]^{-1}[y,\,s_x]\in \gamma_3(R)$.  Working modulo ${\mathfrak  r\mathfrak f}^3,$ we have
\begin{align*}
w-1& \equiv m(x-1)(y-1)-m(y-1)(x-1)-([x,\,s_y]-1)+([y,\,s_x]-1)\\ & \equiv (x-1)(y^m-1)-(y-1)(x^m-1)-([x,\,s_y]-1)+([y,\,s_x]-1)\\ & \equiv
(x-1)(s_y-1)-(y-1)(s_x-1)-([x,\,s_y]-1)+([y,\,s_x]-1)\\ & \equiv (s_y-1)(x-1)-(s_x-1)(y-1)\equiv 0,
\end{align*}
i.e.,  $w\in F(3,\,R)$. Hence, the map $q$ is an epimorphism.
\par\vspace{.25cm}
We next  show that the left hand vertical map in the diagram  (\ref{foxdia}), namely,
$$\frac{\gamma_3(R)}{[[R\cap \gamma_2(F),\,R],\,R\cap \gamma_2(F)]\gamma_4(R)}\to \frac{{\mathfrak  r}^3\mathbb Z[F]}{{\mathfrak  r}^3\mathbb Z[F]\cap {\mathfrak  r\mathfrak f}^3}$$
is a monomorphism, i.e.,
$$
\gamma_3(R)\cap (1+{\mathfrak  r}^3\mathbb Z[F]\cap {\mathfrak  r\mathfrak f}^3)=[[R\cap \gamma_2(F),\,R],\,R\cap \gamma_2(F)]\gamma_4(R).
$$
Clearly,
$$
\gamma_3(R)\cap (1+{\mathfrak  r}^3\mathbb Z[F]\cap {\mathfrak  r\mathfrak f}^3)=\gamma_3(R)\cap (1+\mathfrak r^3\cap {\mathfrak  r\mathfrak f}^3).
$$
Next we will use the following identification (see \cite{KKV}):
$$
\mathfrak r^2\cap {\mathfrak f}^3=\mathfrak r^3+\mathfrak r\Delta(R\cap \gamma_2(F))+\Delta(\gamma_2(R)\cap \gamma_3(F)).
$$
where $\Delta(H) $ denotes the augmentation ideal of the group ring $\mathbb Z[H]$.
We have
\begin{align*}
& \mathfrak r^3\cap {\mathfrak  r\mathfrak f}^3=\mathfrak r(\mathfrak r^2\cap {\mathfrak f}^3)\quad \text{(recall the quote preceding Lemma \ref{fg1})}\\&= \mathfrak r(\mathfrak r\Delta(R\cap \gamma_2(F))+\Delta(R\cap \gamma_2(F))\mathfrak r+\Delta(\gamma_2(R)\cap \gamma_3(F))+\mathfrak r^3)\\&=\mathfrak r(\mathfrak r\Delta(R\cap \gamma_2(F))+\Delta(R\cap \gamma_2(F))\mathfrak r+\Delta([R\cap \gamma_2(F),R])+\mathfrak r^3)\\&=
 \mathfrak r\Delta(R\cap \gamma_2(F))\mathfrak r+\mathfrak r^2\Delta(R\cap \gamma_2(F))+\mathfrak r^4,
\end{align*}
Corollary \ref{corfrf} implies that the right hand vertical arrow in (\ref{foxdia}) is a monomorphism. Hence
$$
\frac{F(3,\,R)}{G(3,\,R)}\simeq L_1{\sf SP}^2(I/J)=L_1{\sf SP}^2\left(\frac{R\cap \gamma_2(F)}{\gamma_2(R)(R\cap \gamma_3(F))}\right)
$$
and the proof part (a) of Theorem \ref{fox} is complete.
\par\vspace{.5cm}\noindent
 Part (b) follows from the part (a) together with the lifting of elements from $L_1{\sf SP}^2(I/J)$ described above.
\par\vspace{.5cm}
\section{Proof of theorem \ref{foxlimit}}
By Theorem 2,
$$
\ilimit \frac{F(3,R)}{G(3,R)}=\ilimit L_1{\sf SP}^2\left(\frac{R\cap \gamma_2(F)}{\gamma_2(R)(R\cap \gamma_3(F))}\right).
$$
In order to study right hand limit, consider  the following diagram with exact rows and columns
\begin{equation}\label{diaw}
\xyma{L_1{\sf SP}^2(G_{ab})\ar@{>->}[d]\ar@{=}[r] & L_1{\sf SP}^2(G_{ab})\ar@{>->}[d]\\
\frac{R\cap \gamma_2(F)}{\gamma_2(R)(R\cap \gamma_3(F))}\ar@{->}[d]\ar@{>->}[r] & \frac{\gamma_2(F)}{\gamma_2(R)\gamma_3(F)}\ar@{->>}[r]\ar@{->}[d] & \gamma_2(G)/\gamma_3(G)\ar@{>->}[d]\\
\frac{{\mathfrak r}\cap {\mathfrak f}^2+{\mathfrak f}^3}{\mathfrak f\mathfrak r+{\mathfrak f}^3}\ar@{>->}[r]& \frac{{\mathfrak f}^2}{\mathfrak f\mathfrak r+{\mathfrak f}^3}\ar@{->>}[r]\ar@{->>}[d] & {\mathfrak g}^2/{\mathfrak g}^3\ar@{->>}[d]\\
& {\sf SP}^2(G_{ab}) \ar@{=}[r] & {\sf SP}^2(G_{ab}). }
\end{equation}
The middle vertical sequence in (\ref{diaw}) is the sequence (\ref{l1sp}) (for $I=F_{ab}$, $J=R/(R\cap \gamma_2(F))$):
$$
\xyma{L_1{\sf SP}^2(G_{ab})\ar@{>->}[r] & \frac{\Lambda^2(F_{ab})}{\Lambda^2(R/(R\cap \gamma_2(F)))}\ar@{=}[d]\ar@{->}[r] & F_{ab}\otimes G_{ab}\ar@{=}[d]\ar@{->>}[r] & {\sf SP}^2(G_{ab})\\ & \frac{\gamma_2(F)}{\gamma_2(R)\gamma_3(F)}\ar@{->}[r] &  \frac{{\mathfrak f}^2}{\mathfrak f\mathfrak r+{\mathfrak f}^3}}
$$
The right hand vertical map $\gamma_2(G)/\gamma_3(G)\to {\mathfrak g}^2/{\mathfrak g}^3$ is a monomorphism, since $D_3(G)=\gamma_3(G)$.
Let us set
$$
K:=\frac{R\cap \gamma_2(F)}{\gamma_2(R)(R\cap \gamma_3(F))},\ L:=\frac{{\mathfrak  r}\mathbb Z[F]\cap {\mathfrak f}^2+{\mathfrak f}^3}{\mathfrak f\mathfrak r+{\mathfrak f}^3}.
$$
Then  we have the following  short exact sequence:
$$
0\to L_1{\sf SP}^2(G_{ab})\to K\to L\to 0.
$$
Observe that, the monomorphism $L_1{\sf SP}^2(G_{ab})\hookrightarrow K$ implies that the induced map
$$
{\sf Tor}(L_1{\sf SP}^2(G_{ab}), \,L_1{\sf SP}^2(G_{ab}))\to {\sf Tor}(K,\,K)
$$
 is also a monomorphism. Therefore, the induced maps
\begin{align*}
& L_1{\sf SP}^2(L_1{\sf SP}^2(G_{ab}))\to L_1{\sf SP}^2(K),\\
& L_1\Lambda^2(L_1{\sf SP}^2(G_{ab}))\to L_1\Lambda^2(K)
\end{align*}
 are also monomorphisms, since $L_1{\sf SP}^2$ and $L_1\Lambda^2$ are subfunctors of the ${\sf Tor}$-functor. Since the functor ${\sf Tor}(-,\,\mathbb Z/2)$ is left-exact, the sequence (\ref{l1to}) implies that the induced map
$$
L_1\tilde\otimes^2(L_1{\sf SP}^2(G_{ab}))\to L_1\tilde\otimes^2(K)
$$
is a monomorphism. Now the sequence (\ref{les}) implies that we have the inclusion
$$
\frac{L_1{\sf SP}^2(K)}{L_1{\sf SP}^2(L_1{\sf SP}^2(G_{ab}))}\hookrightarrow \frac{{\sf Tor}(K,\,K)}{{\sf Tor}(L_1{\sf SP}^2(G_{ab}),\,L_1{\sf SP}^2(G_{ab}))}.
$$
We assert that
\begin{equation}\label{vantor}
\ilimit \frac{{\sf Tor}(K,\,K)}{{\sf Tor}(L_1{\sf SP}^2(G_{ab}),\,L_1{\sf SP}^2(G_{ab}))}=0.
\end{equation}
The natural consequence of the above  vanishing  of $\ilimit$ is that we have  the isomorphism
$$
L_1{\sf SP}^2(L_1{\sf SP}^2(G_{ab}))=\ilimit L_1{\sf SP}^2(K)
$$
and Theorem \ref{foxlimit} will follow.
\par\vspace{.25cm}It remains to establish the vanishing result (\ref{vantor}).
Observe that we have an exact sequence
$$
0\to {\sf Tor}(L_1{\sf SP}^2(G_{ab}),\, L_1{\sf SP}^2(G_{ab}))\to {\sf Tor}(L_1{\sf SP}^2(G_{ab}), \,K)\to {\sf Tor}(L_1{\sf SP}^2(G_{ab}), \,L).
$$
Since $L\subset \frac{{\mathfrak f}^2}{{\mathfrak r\mathfrak f}+{\mathfrak f}^3}=G_{ab}\otimes F_{ab},$ there is a natural inclusion
$$
{\sf Tor}(L_1{\sf SP}^2(G_{ab}), \,L)\hookrightarrow {\sf Tor}(L_1{\sf SP}^2(G_{ab}), \,G_{ab}\otimes F_{ab}).
$$
The short exact sequence
$$
0\to \frac{R}{R\cap \gamma_2(F)}\otimes F_{ab}\to F_{ab}\otimes F_{ab}\to G_{ab}\otimes F_{ab}\to 0,
$$ on
tensoring with $L_1{\sf SP}^2(G_{ab})$, gives an inclusion
$$
{\sf Tor}(L_1{\sf SP}^2(G_{ab}), \,G_{ab}\otimes F_{ab})\hookrightarrow {\sf Tor}(L_1{\sf SP}^2(G_{ab}), \,G_{ab}\otimes F_{ab})\otimes\frac{R}{R\cap \gamma_2(F)}\otimes F_{ab}.
$$
Since the representation ${\sf Tor}(L_1{\sf SP}^2(G_{ab}), \,G_{ab}\otimes F_{ab})\otimes\frac{R}{R\cap \gamma_2(F)}\otimes F_{ab}$ is mono-additive
[one can easily check that, for any representation $\mathcal F$, the representation $\mathcal F\otimes F_{ab}$ is mono-additive], it follows that
$$
\ilimit {\sf Tor}(L_1{\sf SP}^2(G_{ab}), L)=0
$$
and consequently
$$
{\sf Tor}(L_1{\sf SP}^2(G_{ab}), \,L_1{\sf SP}^2(G_{ab}))=\ilimit {\sf Tor}(L_1{\sf SP}^2(G_{ab}),\, K).
$$
A similar argument  shows that $\ilimit {\sf Tor}(L,\,K)=0$ and hence
$$
\ilimit {\sf Tor}(L_1{\sf SP}^2(G_{ab}), K)=\ilimit {\sf Tor}(K,\,K).
$$
Since $\ilimit^1 {\sf Tor}(L_1{\sf SP}^2(G_{ab}), L_1{\sf SP}^2(G_{ab}))=0$, we conclude  (\ref{vantor}). $\Box$
\par\vspace{.5cm}
\noindent{\bf Remark.} Theorem \ref{fox} implies that there is a natural exact sequence
$$
0\to L_1{\sf SP}^2(K)\to \frac{G(2,\,R)}{G(3,\,R)}\to \frac{{\mathfrak  r\mathfrak f}^2}{{\mathfrak  r\mathfrak f}^3},
$$
The quotient
$$
\frac{{\mathfrak  r\mathfrak f}^2}{{\mathfrak r\mathfrak f}^3}\simeq \frac{\mathfrak  r\mathbb Z[F]}{\mathfrak  r\mathfrak f}\otimes \frac{{\mathfrak f}^2}{{\mathfrak f}^3}
$$
is a mono-additive representation. Hence,
$$
\ilimit \frac{G(2,\,R)}{G(3,\,R)}=L_1{\sf SP}^2(L_1{\sf SP}^2(G_{ab})).
$$
\par\vspace{.5cm}
\noindent{\bf Remark.} The middle vertical sequence in (\ref{diaw}) implies that, if $G_{ab}$ is torsion-free, then the group $\frac{\gamma_2(F)}{\gamma_2(R)\gamma_3(F)}$ is torsion-free as well. Since the quotient $\frac{R\cap \gamma_2(F)}{\gamma_2(R)(R\cap \gamma_3(F))}$ is a subgroup of $\frac{\gamma_2(F)}{\gamma_2(R)\gamma_3(F)}$, we conclude that, if $G_{ab}$ is torsion-free, then $F(3,\,R)=G(3,\,R),$ by Theorem 2.
\par\vspace{.5cm}
\centerline{\bf Acknowledgement}\par\vspace{.25cm}
The research is supported by the Russian Science Foundation grant
N 16-11-10073. The authors are thankful to Harish-Chandra Research Institute, Allahabad, for the warm hospitality provided to them during their visit in February 2017. \par
\newpage
\par\vspace{1cm}\noindent
Roman Mikhailov\\ St Petersburg
Department of Steklov Mathematical Institute,\\Chebyshev
Laboratory, St Petersburg State University\\ 14th Line, 29b, Saint
Petersburg 199178 Russia\\
and \\School of Mathematics, Tata Institute of Fundamental Research\\
Mumbai 400005, India\\
Email:\ romanvm@mi.ras.ru\par\vspace{.5cm}\noindent
Inder Bir S. Passi\\ Centre for Advanced Study in Mathematics\\ Panjab
University, Sector 14, Chandigarh 160014 India\\ and \\ Indian
Institute of Science Education and Research, Mohali (Punjab)
140306 India\\
Email:\ ibspassi@yahoo.co.in
\end{document}